\documentclass[11pt]{article}
\usepackage{amsmath}
\usepackage{amssymb}
\usepackage{amsthm}
\usepackage[usenames]{color}
\usepackage{amscd}
\usepackage{dsfont}
\usepackage{indentfirst}

\usepackage[colorlinks=true,linkcolor=blue,filecolor=red,
citecolor=webgreen]{hyperref}
\definecolor{webgreen}{rgb}{0,.5,0}

\numberwithin{equation}{section}

\def\C{{\mathds{C}}}

\def\R{{\mathbb{R}}}
\def\N{{\mathds{N}}}
\def\Z{{\mathds{Z}}}
\def\P{{\mathds{P}}}
\def\1{{\bf 1}}

\def\id{\operatorname{id}}

\newtheorem{theorem}{Theorem}[section]

\newtheorem{lemma}[theorem]{Lemma}
\newtheorem{cor}[theorem]{Corollary}

\newtheorem{prop}[theorem]{Proposition}

\begin{document}

\title{{\bf On certain sums of arithmetic functions involving the gcd and lcm of two positive integers}}
\author{Randell Heyman \\
School of Mathematics and Statistics \\
University of New South Wales \\
Sydney, Australia \\
E-amail: {\tt randell@unsw.edu.au} \\ \ \\
L\'aszl\'o T\'oth \\ Department of Mathematics, University of P\'ecs \\
Ifj\'us\'ag \'utja 6, 7624 P\'ecs, Hungary \\ E-mail: {\tt ltoth@gamma.ttk.pte.hu}}
\date{}
\maketitle

\centerline{Results Math. 76 (2021), no. 1, Paper No. 49, 22 pp.}

\begin{abstract} We obtain asymptotic formulas with remainder terms for the hyperbolic summations $\sum_{mn\le x} f((m,n))$ and
$\sum_{mn\le x} f([m,n])$, where $f$ belongs to certain classes of arithmetic functions, $(m,n)$ and $[m,n]$ denoting the gcd and lcm
of the integers $m,n$. In particular, we investigate the functions $f(n)=\tau(n), \log n, \omega(n)$ and  $\Omega(n)$. We also define a common
generalization of the latter three functions, and prove a corresponding result.
\end{abstract}



\section{Introduction}

Let $F:\N^2 \to \C$ be an arithmetic function of two variables. Several asymptotic results
for sums $\sum F(m,n)$ with various bounds of summation are given in the literature. The usual `rectangular' summations are of form $\sum_{m\le x, \, n\le y} F(m,n)$, in particular
with $x=y$. The `triangular' summations can be written as $\sum_{n \le x} \sum_{m\le n} F(m,n)$. Note that if the function $F$ is symmetric in
the variables, then
\begin{equation*}
\sum_{m,n\le x} F(m,n) = 2 \sum_{n\le x} \sum_{m\le n} F(m,n) -\sum_{n\le x} F(n,n).
\end{equation*}

The `hyperbolic' summations have the shape $\sum_{mn\le x} F(m,n)$, the sum being over the Dirichlet region $\{(m,n)\in \N^2: mn\le x\}$. Hyperbolic summations have been less studied than rectangular and triangular summations and it is hyperbolic summations that are estimated in this paper.

We mention a
few examples for functions $F$ involving the greatest common divisor (gcd) and the least common multiple (lcm) of integers.
If $F(m,n)=(m,n)$, the gcd of $m$ and $n$, then
\begin{equation} \label{gcd_m_n}
\sum_{m,n\le x} (m,n)= \frac{x^2}{\zeta(2)}\left(\log x+ 2\gamma
-\frac1{2}-\frac{\zeta(2)}{2}- \frac{\zeta'(2)}{\zeta(2)} \right) +
O\left(x^{1+\theta+\varepsilon}\right),
\end{equation}
holds for every $\varepsilon>0$, where $\zeta$ is the Riemann zeta function, $\zeta'$ is its derivative, $\gamma$ is Euler's constant,
and $\theta$ denotes the exponent appearing in Dirichlet's divisor problem. Furthermore,
\begin{equation} \label{gcd_mn}
\sum_{mn\le x} (m,n) = \frac1{4\zeta(2)} x(\log x)^2 + c_1 x\log x + c_2x + O\left( x^{\beta} (\log x)^{\beta'}\right),
\end{equation}
where $c_1,c_2$ are explicit constants, and $\beta=547/832 \doteq 0.65745$1, $\beta'=26947/8320 \doteq 3.238822$. Estimate \eqref{gcd_m_n} (in the form of
a triangular summation, involving Pillai's arithmetic function) was obtained by Chidambaraswamy and Sitaramachandrarao \cite[Th.\ 3.1]{ChiSit1985} using
elementary arguments. Formula \eqref{gcd_mn} was deduced applying analytic methods by Kr\"atzel, Nowak and T\'oth \cite[Th.\ 3.5]{KNT2012}.

If $F(m,n)=[m,n]$, the lcm  of $m$ and $n$, then we have
\begin{equation} \label{lcm_m_n}
\sum_{m,n\le x} [m,n] = \frac{\zeta(3)}{4\zeta(2)} x^4 + O\left(x^3(\log x)^{2/3}(\log \log x)^{1/3}\right),
\end{equation}
established by Bordell\`{e}s \cite[Th.\ 6.3]{Bor2007} with a slightly weaker error term. The error in \eqref{lcm_m_n} comes
from Liu's \cite{Liu2016} improvement for the error term by Walfisz on $\sum_{n\le x} \varphi(n)$, where $\varphi$ is Euler's function.
Also see \cite{HLT2020}.

We also have
\begin{equation} \label{sum_gcd_per_lcm}
\sum_{m,n\le x} \frac{(m,n)}{[m,n]} = 3 x + O\left((\log x)^2\right),
\end{equation}
obtained by Hilberdink, Luca and T\'oth \cite[Th.\ 5.1]{HLT2020}.

For other related asymptotic results for functions involving the gcd and lcm of two (and several) integers see the papers
\cite{Bor2007,HLT2020,KNT2012,Tot2010,TotZha2018} and their references. For summations of some other functions of two variables
see \cite{CK2012,SL2020,Ush2016}.

We remark that for any arithmetic function $f:\N \to \C$ of one variable,
\begin{equation} \label{sum_gcd_mn}
\sum_{mn\le x} f((m,n)) = \sum_{k\le x} G_f(k),
\end{equation}
where $G_f(k)=\sum_{mn=k} f((m,n))$. Hence to find estimates for \eqref{sum_gcd_mn} is, in fact, a one variable summation problem.
Formula \eqref{gcd_mn} represents its special case when $f(k)=k$ ($k\in \N$). The sum of divisors function $f(k)=\sigma(k)$,
giving an estimate similar to \eqref{gcd_mn}, with the same error term, was considered in paper \cite{KNT2012}. The case of the
divisor function $f(k)=\tau(k)$ was discussed by Heyman \cite{Hey}, obtaining an asymptotic formula with error term $O(x^{1/2})$ by using
elementary estimates. However, for every $k\in \N$,
\begin{equation*}
\sum_{mn=k} \tau((m,n)) = \sum_{abc^2=k} 1 =: \tau(1,1,2;k),
\end{equation*}
which follows from the general arithmetic identities for $G_f(k)$ in Proposition \ref{Prop_gen}, see in particular \eqref{Gf2}, and the summation of
the divisor function $\tau(1,1,2;k)$ is well known in the literature. See, e.g., Kr\"atzel \cite[Ch.\ 6]{Kra1988}. The best known related error term,
to our knowledge, is $O(x^{63/178+\varepsilon})$, with $63/178\doteq 0.353932$, given by Liu \cite{Liu2010} using deep analytic methods.

We deduce the following result.
\begin{theorem} We have
\begin{equation} \label{tau_gcd}
\sum_{mn\le x} \tau((m,n)) =  \zeta(2) x \left(\log x+ 2\gamma-1+2\frac{\zeta'(2)}{\zeta(2)}\right) + \zeta^2(1/2)x^{1/2} +
O(x^{63/178+\varepsilon}).
\end{equation}
\end{theorem}

In an analogous manner to \eqref{sum_gcd_mn}, let us define
\begin{equation} \label{sum_lcm_mn}
\sum_{mn\le x} f([m,n]) = \sum_{k\le x} L_f(k),
\end{equation}
where $L_f(k)=\sum_{mn=k} f([m,n])$.

We remark that if $F$ is an arbitrary arithmetic function of two variables, then the one variable function
\begin{equation*}
\widetilde{F}(k)= \sum_{mn=k} F(m,n)
\end{equation*}
is called the convolute of $F$. The function $F$ of two variables is said to be multiplicative if $F(m_1m_2,n_1n_2)=F(m_1,n_1)F(m_2,n_2)$
provided that $(m_1n_1,m_2n_2)=1$. If $F$ is multiplicative, then $\widetilde{F}$ is also multiplicative. See Vai\-dy\-anathas\-wamy
\cite{Vai1931}, T\'oth \cite[Sect.\ 6]{Tot2014}. The functions $G_f$ and $L_f$ of above are special cases of this general concept.
If $f$ is multiplicative, then $G_f(k)$ and $L_f(k)$ are multiplicative as well.

In the present paper we deduce simple arithmetic representations of the functions $G_f(k)$ and $L_f(k)$ (Proposition \ref{Prop_gen}),
and establish new asymptotic estimates for sums of type \eqref{sum_gcd_mn} and \eqref{sum_lcm_mn}. Namely, we give estimates for
$\sum_{mn\le x} f((m,n))$ when $f$ belongs to a wide class of functions (Theorem \ref{Th_1_general}), and obtain better error
terms in the case of a narrower class of functions (Theorem \ref{Th_2_general}).
In particular, we consider the functions $f(n)= \log n, \omega(n)$ and  $\Omega(n)$ (Corollary \ref{Cor_log}). Actually, we define a common generalization
of these three functions and prove a corresponding result (Corollary \ref{Cor_f_S_eta}). We also point out the case of the function $f(n)=1/n$,
the related result on $\sum_{mn\le x} (m,n)^{-1}$ (Corollary \ref{Cor_1_per_(a,b)}) being strongly connected with the sum $\sum_{mn\le x} [m,n]$
(Theorem \ref{Theorem_sum_lcm_m_n}). Furthermore, we deduce estimates for the sums $\sum_{mn\le x} f([m,n])$  in the cases of $f(n)=\log n,
\omega(n), \Omega(n)$ (Theorems \ref{Th_log_lcm}, \ref{Th_omega_lcm}, \ref{Th_Omega_lcm}) and $f(n)=\tau(n)$ (Theorem \ref{Th_tau_lcm}), respectively. Finally we obtain a formula for $\sum_{mn\le x} (m,n)[m,n]^{-1}$
(Theorem \ref{Th_gcd_per_lcm}). The proofs are given in Section \ref{Section_Proofs}.

Throughout the paper we use the following notation: $\N=\{1,2,\ldots\}$; $\P=\{2,3,5,\ldots\}$ is the set of primes;
$n=\prod_p p^{\nu_p(n)}$ is the prime power factorization of $n\in \N$, the product being over $p\in \P$, where all but
a finite number of the exponents $\nu_p(n)$ are zero; $\tau(n)=\sum_{d\mid n}1$ is the divisor function; $\1(n)=1$, $\id(n)=n$ ($n\in \N$);
$\mu$ is the M\"obius function; $\omega(n)= \# \{p: \nu_p(n)\ne 0\}$; $\Omega(n)= \sum_p \nu_p(n)$; $\kappa(n)=\prod_{\nu_p(n)\ne 0} p$ is the
squarefree kernel of $n$; $*$ is the Dirichlet convolution of arithmetic functions; $\zeta$ is the Riemann zeta function, $\zeta'$ is its derivative, $\pi(x)=\sum_{p\le x} 1$; $\gamma$ is Euler's constant.

\section{Main results}

Useful arithmetic representations of the functions $G_f(n)=\sum_{ab=n} f((a,b))$ and $L_f(n)=\sum_{ab=n} f([a,b])$, already defined in
the Introduction, are given by the next result.

\begin{prop} \label{Prop_gen}
Let $f$ be an arbitrary arithmetic function. Then for every $n\in \N$,
\begin{align} \label{Gf1}
G_f(n)= & \sum_{a^2 b^2 c=n} f(a) \mu(b) \tau(c) \\
= & \sum_{a^2c=n} (f*\mu)(a)\tau(c) \label{Gf2} \\
= & \sum_{a^2c=n} f(a)\, 2^{\omega(c)}, \label{Gf3}
\end{align}
and
\begin{align} \label{Lf1}
L_f(n)= & \sum_{a^2 b^2 c=n} f(n/a) \mu(b) \tau(c) \\
= & \sum_{a^2c=n} f(ac)\, 2^{\omega(c)}. \label{Lf2}
\end{align}

If $f$ is additive, then for every $n\in \N$,
\begin{equation} \label{GLf_addit}
L_f(n)= 2(f*\1)(n)  - G_f(n).
\end{equation}

If $f$ is completely additive, then for every $n\in \N$,
\begin{equation} \label{GLf}
L_f(n)= f(n)\tau(n) - G_f(n).
\end{equation}
\end{prop}

In terms of formal Dirichlet series, identities \eqref{Gf1}, \eqref{Gf2} and \eqref{Gf3} show that for every arithmetic function $f$,
\begin{equation*}
\sum_{n=1}^{\infty} \frac{G_f(n)}{n^z} =  \frac{\zeta^2(z)}{\zeta(2z)} \sum_{n=1}^{\infty} \frac{f(n)}{n^{2z}}.
\end{equation*}

See \cite[Prop.\ 5.1]{KNT2012} for a similar formula on the sum $\sum_{d_1\cdots d_k=n} g((d_1,\ldots,d_k))$, where $k\in \N$ and $g$ is an arithmetic function.

Our first asymptotic formula applies to every function $f$ satisfying a condition on its order of magnitude.

\begin{theorem} \label{Th_1_general}
Let $f$ be an arithmetic function such that $f(n) \ll n^\beta (\log n)^{\delta}$, as $n\to \infty$, for some fixed $\beta, \delta \in \R$ with $\beta<1$.
Then
\begin{equation} \label{f_asymp}
\sum_{mn\le x} f((m,n)) = x(C_f\log x+ D_f) + R_f(x),
\end{equation}
where the constants $C_f$ and $D_f$ are given by
\begin{equation*}
C_f= \frac1{\zeta(2)} \sum_{n=1}^{\infty} \frac{f(n)}{n^2},
\end{equation*}
\begin{equation*}
D_f= \frac1{\zeta(2)} \left(C \sum_{n=1}^{\infty} \frac{f(n)}{n^2} - 2 \sum_{n=1}^{\infty} \frac{f(n)\log n}{n^2}\right),
\end{equation*}
with $C$ defined by
\begin{equation} \label{def_C}
C= 2\gamma -1-\frac{2\zeta'(2)}{\zeta(2)},
\end{equation}
and the error term is
\begin{equation*}
R_f(x)\ll \begin{cases} x^{(\beta+1)/2} (\log x)^{\delta+1}, & \text{ if $0<\beta <1$ or $\beta=0$, $\delta \ne -1$}, \\
x^{1/2} \log \log x, & \text{ if $\beta=0$, $\delta = -1$}, \\
x^{1/2} \lambda(x), & \text{ if $\beta < 0$},
\end{cases}
\end{equation*}
where
\begin{equation} \label{lambda}
\lambda(x):= e^{-c(\log x)^{3/5}(\log \log x)^{-1/5}},
\end{equation}
with some constant $c>0$.

The error term $R_f(x)$ can be improved assuming that the Riemann Hypothesis (RH) is true. For example, let $\varrho=221/608\doteq 0.363486$.
If $\beta, \delta \in \R$ and $\beta < 2\varrho -1\doteq -0.273026$, then $R_f(x)\ll x^{\varrho+\varepsilon}$.
\end{theorem}

Theorem \ref{Th_1_general} applies, e.g., to the functions $f(n)=n^{\beta}$ (with $\beta <1$, $\delta=0$),
$f(n)=(\log n)^{\delta}$ (with $\beta=0$, $\delta \in \R$), $f(n)=\tau^k(n)$ ($k\in \N$, with $\beta=k\varepsilon$, $\varepsilon >0$
arbitrary small, $\delta=0$), $f(n)=\omega(n)$ or $ \Omega(n)$ (with $0<\beta=\varepsilon$ arbitrary small, $\delta=0$). We point out the case
of the function $f(n)=n^{-1}$.

\begin{cor} \label{Cor_1_per_(a,b)} We have
\begin{equation} \label{estimate_1_per_(a,b)}
\sum_{mn\le x} \frac1{(m,n)} = \frac{\zeta(3)}{\zeta(2)}x (\log x+ D) + O(x^{1/2} \lambda(x)),
\end{equation}
where
\begin{equation*}
D= 2\gamma-1 -2\frac{\zeta'(2)}{\zeta(2)}+2\frac{\zeta'(3)}{\zeta(3)},
\end{equation*}
and $\lambda(x)$ is defined by \eqref{lambda}.  If RH is true, then the error term is $O(x^{\varrho+\varepsilon})$,
where $\varrho$ is given in Theorem \ref{Th_1_general}.
\end{cor}

However, for some special functions asymptotic formulas with more terms or with better unconditional errors can be obtained. See,
e.g. \eqref{tau_gcd}, namely the case $f(n)=\tau(n)$ and our next results.

Let $f$ be a function such that $(\mu *f)(n)=0$ for all $n\ne p^\nu$ ($n$ is not a prime power), $(\mu *f)(p^\nu)=g(p)$ does not depend
on $\nu$ and $g(p)$ is sufficiently small for the primes $p$. More exactly, we have the next result.

\begin{theorem} \label{Th_2_general}
Let $f$ be an arithmetic function such that there exists a subset $Q$ of the set of primes $\P$ and there exists a subset $S$ of $\N$ with
$1\in S$, satisfying the following properties:

i) $(\mu*f)(n)=0$ for all $n\ne p^\nu$, where $p\in Q$ and $\nu \in S$,

ii) $(\mu*f)(p^\nu)= g(p)$, depending only on $p$, for all prime powers $p^\nu$ with $p\in Q$, $\nu \in S$.

iii) $g(p) \ll (\log p)^{\eta}$, as $p\to \infty$, where $\eta \ge 0$ is a fixed real number.

Then
for the error term in \eqref{f_asymp} we have $R_f(x) \ll x^{1/2} (\log x)^{\eta}$. Furthermore, the constants $C_f$ and $D_f$ can be given as
\begin{equation*}
C_f = \sum_{p\in Q} g(p) \sum_{\nu \in S} \frac1{p^{2\nu}},
\end{equation*}
\begin{equation*}
D_f = (2\gamma-1) C_f - 2 \sum_{p\in Q} g(p)\log p \sum_{\nu \in S} \frac{\nu}{p^{2\nu}}.
\end{equation*}
\end{theorem}

The prototype of functions $f$ to which Theorem \ref{Th_2_general} applies is the function $f_{S,\eta}$ implicitly defined by
\begin{equation} \label{f_S_eta_cond}
h_{S,\eta}(n):= (\mu*f_{S,\eta})(n)= \begin{cases} (\log p)^{\eta}, & \text{ if $n=p^\nu$ a prime power with $\nu \in S$},\\ 0, & \text{ otherwise},
\end{cases}
\end{equation}
where $1\in S\subseteq \N$, $\eta \ge 0$ is real and $Q=\P$. It is possible to consider the corresponding generalization with $Q\subset \P$, as well.
By M\"{o}bius inversion we obtain that for $n=\prod_p p^{\nu_p(n)}\in \N$,
\begin{equation*}
f_{S,\eta}(n)= \sum_{d\mid n} h_{S,\eta}(d)= \sum_{p\mid n} (\log p)^{\eta} \# \{\nu: 1\le \nu\le \nu_p(n), \nu \in S \},
\end{equation*}
where $f_{S,\eta}(1)=0$ (empty sum).

Let $S=\N$. Then
\begin{equation*}
f_{\N,\eta}(n):=  \sum_{p\mid n} \nu_p(n) (\log p)^{\eta},
\end{equation*}
which gives for $\eta=1$, $f_{\N,1}(n)=\log n$, while $h_{\N,1}(n)=\Lambda(n)$ is the von Mangoldt function. If $\eta=0$, then $f_{\N,0}(n) =\Omega(n)$.

Now let $S=\{1\}$. Then
\begin{equation*}
f_{\{1\},\eta}(n):=  \sum_{p\mid n} (\log p)^{\eta},
\end{equation*}
and if $\eta=0$, then $f_{\{1\},0}(n) =\omega(n)$. If $\eta=1$, then $f_{\{1\},1}(n) =\log \kappa(n)$, where $\kappa(n)=\prod_{p\mid n} p$.
Note that $\sum_{n\le x} h_{\{1\},1}(n)=\sum_{p \le x} \log p = \theta(x)$ is the Chebyshev theta function.

The functions $f_{S,\eta}(n)$ and $h_{S,\eta}(n)$ have not been studied in the literature, as far as we know.

According to \eqref{f_S_eta_cond}, the conditions of Theorem \ref{Th_2_general} are satisfied and we deduce the next result.

\begin{cor} \label{Cor_f_S_eta}  If $1\in S\subseteq \N$ and $\eta \ge 0$ is a real number, then
\begin{equation*}
\sum_{mn\le x} f_{S,\eta}((m,n)) = x(C_{f_{S,\eta}} \log x + D_{f_{S,\eta}}) + O(x^{1/2} (\log x)^{\eta}),
\end{equation*}
where the constants $C_{f_{S,\eta}}$ and $D_{f_{S,\eta}}$ are given by
\begin{equation*}
C_{f_{S,\eta}} = \sum_p (\log p)^{\eta} \sum_{\nu \in S} \frac1{p^{2\nu}},
\end{equation*}
\begin{equation*}
D_{f_{S,\eta}} = (2\gamma-1)C_{f_{S,\eta}} - 2 \sum_p (\log p)^{\eta+1} \sum_{\nu \in S} \frac{\nu}{p^{2\nu}}.
\end{equation*}
\end{cor}

In the special cases mentioned above we obtain the following results.

\begin{cor} \label{Cor_log} We have
\begin{equation} \label{log_asymp}
\sum_{mn\le x} \log (m,n) = x(C_{\log} \log x+ D_{\log}) + O(x^{1/2} \log x),
\end{equation}
\begin{equation} \label{log_kappa_asymp}
\sum_{mn\le x} \log \kappa((m,n)) = x(C_{\log \kappa} \log x+ D_{\log \kappa}) + O(x^{1/2} \log x),
\end{equation}
\begin{equation} \label{omega_asymp}
\sum_{mn\le x} \omega((m,n)) = x(C_{\omega}\log x+ D_{\omega}) + O(x^{1/2}),
\end{equation}
\begin{equation} \label{Omega_asymp}
\sum_{mn\le x} \Omega((m,n)) = x(C_{\Omega}\log x+ D_{\Omega}) + O(x^{1/2}),
\end{equation}
where
\begin{equation} \label{C_D_log}
C_{\log}= - \frac{\zeta'(2)}{\zeta(2)} = \sum_p \frac{\log p}{p^2-1} \doteq 0.569960, \quad D_{\log}= -\frac{\zeta'(2)}{\zeta(2)} \left(2\gamma-1-2\frac{\zeta'(2)}{\zeta(2)} + 2\frac{\zeta''(2)}{\zeta'(2)}\right),
\end{equation}
\begin{equation*}
C_{\log \kappa}= \sum_p \frac{\log p}{p^2}\doteq 0.493091,  \quad D_{\log \kappa}= (2\gamma-1)\sum_p \frac{\log p}{p^2} - 2 \sum_p \frac{(\log p)^2}{p^2},
\end{equation*}
\begin{equation} \label{C_D_omega}
C_{\omega}= \sum_p \frac1{p^2}\doteq 0.452247,  \quad D_{\omega}= (2\gamma-1)\sum_p \frac1{p^2} - 2 \sum_p \frac{\log p}{p^2},
\end{equation}
\begin{equation} \label{C_D_Omega}
C_{\Omega}= \sum_p \frac1{p^2-1}\doteq 0.551693,  \quad D_{\Omega}= (2\gamma-1)\sum_p \frac1{p^2-1} - 2 \sum_p \frac{p^2\log p}{(p^2-1)^2}.
\end{equation}
\end{cor}

We deduce by \eqref{log_asymp} and \eqref{log_kappa_asymp} that $\prod_{mn\le x} (m,n) \sim x^{C_{\log}x}$ and $\prod_{mn\le x} \kappa((m,n))
\sim x^{C_{\log \kappa}x}$, as $x\to \infty$.

Now consider the functions given by $L_f(n)=\sum_{ab=n} f([a,b])$. If $f=\id$, then $L_{\id}(n)=\sum_{ab=n} [a,b] = n \sum_{ab=n} (a,b)^{-1}$.
The next result follows from Corollary \ref{Cor_1_per_(a,b)} by partial summation. It may be compared to
estimates \eqref{gcd_m_n}, \eqref{gcd_mn} and \eqref{lcm_m_n}.

\begin{theorem} \label{Theorem_sum_lcm_m_n} We have
\begin{equation*}
\sum_{mn\le x} [m,n] = \frac{\zeta(3)}{2\zeta(2)}x^2 (\log x+ E) + O(x^{3/2} \lambda(x)),
\end{equation*}
where $\lambda(x)$ is defined by \eqref{lambda}, and
\begin{equation*}
E= 2\gamma-\frac1{2}-2\frac{\zeta'(2)}{\zeta(2)}+ 2\frac{\zeta'(3)}{\zeta(3)}.
\end{equation*}

If RH is true, then the error term is $O(x^{1+\varrho+\varepsilon})$, where $\varrho$ is given in Theorem \ref{Th_1_general}.
\end{theorem}

If the function $f$ is (completely) additive, then identities \eqref{GLf_addit} and \eqref{GLf} can be used to deduce asymptotic estimates 
for $\sum_{n\le x} L_f(n)$.

\begin{theorem} \label{Th_log_lcm} We have
\begin{equation*}
\sum_{mn\le x} \log [m,n] = x(\log x)^2 + (2\gamma-2-C_{\log}) x\log x - (2\gamma-2 + D_{\log})x + O(x^{1/2}\log x),
\end{equation*}
where $C_{\log}$ and $D_{\log}$ are given by \eqref{C_D_log}.
\end{theorem}

As a consequence we deduce that $\prod_{mn\le x} [m,n] \sim x^{x \log x}$ as $x\to \infty$.

\begin{theorem} \label{Th_omega_lcm}  We have
\begin{equation} \label{form_omega_lcm}
\sum_{mn\le x} \omega([m,n]) = 2x (\log x)(\log \log x) +(K_{\omega} - C_{\omega}) x\log x + O(x),
\end{equation}
where $C_{\omega}$ is given by \eqref{C_D_omega} and
\begin{equation} \label{const_K_omega}
K_{\omega}= 2\left(\gamma-1 + \sum_p \left(\log\left(1-\frac1{p}\right) +\frac1{p} \right)\right).
\end{equation}
\end{theorem}

\begin{theorem} \label{Th_Omega_lcm}  We have
\begin{equation} \label{form_Omega_lcm}
\sum_{mn\le x} \Omega([m,n]) = 2x (\log x)(\log \log x) +(K_{\Omega}-C_{\Omega}) x\log x + O(x),
\end{equation}
where $C_{\Omega}$ is given by \eqref{C_D_Omega} and
\begin{equation} \label{const_K_Omega}
K_{\Omega}= 2\left(\gamma-1 + \sum_p \left(\log\left(1-\frac1{p}\right) +\frac1{p-1} \right)\right).
\end{equation}
\end{theorem}

Next we consider the divisor function $f(n)=\tau(n)$.

\begin{theorem} \label{Th_tau_lcm} We have
\begin{equation} \label{tau_lcm}
\sum_{mn\le x} \tau([m,n]) = x(C_1(\log x)^3+ C_2(\log x)^2+ C_3\log x +C_4) + O(x^{1/2+\varepsilon})
\end{equation}
for every $\varepsilon>0$, where
\begin{equation*}
C_1= \frac1{\pi^2}\prod_p \left(1-\frac1{(p+1)^2}\right)\doteq 0.078613,
\end{equation*}
and the constants $C_2,C_3,C_4$ can also be given explicitly.
\end{theorem}

Estimate \eqref{tau_lcm} may be compared to \eqref{tau_gcd} and to that for $\sum_{m,n\le x} \tau([m,n])$. See T\'oth and Zhai
\cite[Th.\ 3.4]{TotZha2018}.

Finally, we deduce the counterpart of formula \eqref{sum_gcd_per_lcm} with hyperbolic summation.

\begin{theorem} \label{Th_gcd_per_lcm} We have
\begin{equation}  \label{final_est}
\sum_{mn\le x} \frac{(m,n)}{[m,n]} =  \frac{\zeta^2(3/2)}{\zeta(3)} x^{1/2} + O((\log x)^3).
\end{equation}
\end{theorem}

\section{Proofs} \label{Section_Proofs}

\begin{proof}[Proof of Proposition {\rm \ref{Prop_gen}}] Group the terms of the sum $G_f(n)=\sum_{ab=n} f((a,b))$
according to the values $(a,b)=d$, where $a=dc, b=de$ with
$(c,e)=1$. We obtain, using the property of the M\"{o}bius $\mu$ function,
\begin{equation*}
G_f(n)= \sum_{\substack{d^2ce=n\\ (c,e)=1}} f(d) =
\sum_{d^2ce=n} f(d) \sum_{\delta \mid (c,e)} \mu(\delta)
\end{equation*}
\begin{equation*}
= \sum_{d^2\delta^2 k\ell =n} f(d) \mu(\delta) = \sum_{d^2\delta^2 t=n} f(d) \mu(\delta) \sum_{k\ell=t} 1
\end{equation*}
\begin{equation*}
= \sum_{d^2\delta^2 t=n} f(d) \mu(\delta) \tau(t),
\end{equation*}
giving \eqref{Gf1}, which can be written as \eqref{Gf2} and \eqref{Gf3} by the definition of the Dirichlet convolution and the identity
$\sum_{\delta^2 t=k} \mu(\delta)\tau(t)=2^{\omega(k)}$.

Alternatively, use the identity $f(n)=\sum_{d\mid n} (f*\mu)(d)$ to deduce that
\begin{equation*}
G_f(n)=\sum_{ab=n} \sum_{d\mid (a,b)} (f*\mu)(d)
=\sum_{ab=n} \sum_{d\mid a, d\mid b} (f*\mu)(d)
\end{equation*}
\begin{equation*}
= \sum_{d^2ce=n} (f*\mu)(d) = \sum_{d^2 t=n} (f*\mu)(d) \sum_{ce=t} 1 = \sum_{d^2 t=n} (f*\mu)(d) \tau(t),
\end{equation*}
giving \eqref{Gf2}.

For $L_f(n)$ use that $[a,b]=ab/(a,b)$ and apply the first method above to deduce \eqref{Lf1} and \eqref{Lf2}.

If $f$ is an additive function, then $f((a,b))+f([a,b])=f(a)+f(b)$ holds for every $a,b\in \N$. To see this, it is enough to consider the case when
$a=p^r$, $b=p^s$ are powers of the same prime $p$. Now, $f(p^{\min(r,s)}) + f(p^{\max(r,s)}) = f(p^r)+f(p^s)$ trivially holds for every $r,s\ge 0$. 
Therefore,
\begin{equation} \label{add}
L_f(n)= \sum_{ab=n} f([a,b])= \sum_{ab=n} f(a) + \sum_{ab=n} f(b) -\sum_{ab=n} f((a,b))
\end{equation}
\begin{equation*}
= 2 (f*\1)(n) -G_f(n),
\end{equation*}
which is \eqref{GLf_addit}.

Finally, to obtain \eqref{GLf}, use that if $f$ is completely additive, then  in \eqref{add} one has
\begin{equation*}
\sum_{ab=n} f(a) + \sum_{ab=n} f(b)  = \sum_{ab=n} f(ab) = f(n) \tau(n).
\end{equation*}
\end{proof}

For the proof of Theorem \ref{Th_1_general} we need the following Lemmas.

\begin{lemma} \label{Lemma_n>x} Let $s,\delta \in \R$ with $s>1$. Then
\begin{equation*}
\sum_{n> x} \frac{(\log n)^{\delta}}{n^s} \ll \frac{(\log x)^{\delta}}{x^{s-1}}.
\end{equation*}
\end{lemma}

\begin{proof} The function $t\mapsto t^{-s}(\log t)^{\delta}$ ($t>x$) is decreasing for large $x$.
By comparing the sum with the corresponding integral we have
\begin{equation*}
\sum_{n>x} \frac{(\log n)^{\delta}}{n^s} \le \int_x^{\infty} \frac{(\log t)^{\delta}}{t^s}\, dt.
\end{equation*}

If $\delta < 0$, then trivially,
\begin{equation*}
\int_x^{\infty} \frac{(\log t)^{\delta}}{t^s}\, dt \le (\log x)^{\delta} \int_x^{\infty} \frac1{t^s}\, dt  \ll \frac{(\log x)^{\delta}}{x^{s-1}}.
\end{equation*}

If $\delta>0$, then integrating by parts gives
\begin{equation*}
\int_x^{\infty} \frac{(\log t)^{\delta}}{t^s}\, dt \ll \frac{(\log x)^{\delta}}{x^{s-1}}+
\int_x^{\infty} \frac{(\log t)^{\delta-1}}{t^s}\, dt,
\end{equation*}
and repeated applications of the latter estimate, until the exponent of $\log t$ becomes negative, conclude the proof.
\end{proof}

\begin{lemma} \label{Lemma_n<x} Let $s,\delta \in \R$ with $s>0$. Then
\begin{equation*}
\sum_{2\le n\le x} \frac{(\log n)^{\delta}}{n^s} \ll
\begin{cases}
x^{1-s} (\log x)^{\delta}, & \text{ if $0<s<1$, $\delta \in \R$}, \\
(\log x)^{\delta+1}, & \text{ if $s=1$, $\delta\ne -1$}, \\
\log \log x, & \text{ if $s=1$, $\delta=-1$}, \\
1, & \text{ if $s>1$}.
\end{cases}
\end{equation*}
\end{lemma}

\begin{proof} Let $0<s<1$. If $\delta \ge 0$, then trivially
\begin{equation*}
\sum_{n\le x} \frac{(\log n)^{\delta}}{n^s} \le (\log x)^{\delta} \sum_{n\le x} \frac1 {n^s}
\end{equation*}
and by comparison of the sum with the corresponding integral we have
\begin{equation} \label{known_estimate_s}
\sum_{n\le x} \frac1{n^s} \ll x^{1-s}.
\end{equation}

If $\delta <0$, then write
\begin{equation*}
\sum_{2\le n\le x} \frac{(\log n)^{\delta}}{n^s} =  \sum_{2\le n\le x^{1/2}} \frac{(\log n)^{\delta}}{n^s} + \sum_{x^{1/2}< n\le x}
\frac{(\log n)^{\delta}}{n^s}
\end{equation*}
\begin{equation*}
\ll \sum_{n\le x^{1/2}} \frac1{n^s} + (\log x^{1/2})^{\delta} \sum_{x^{1/2}< n\le x} \frac1{n^s},
\end{equation*}
which is, using again \eqref{known_estimate_s},
\begin{equation*}
\ll x^{(1-s)/2} + (\log x)^{\delta} x^{1-s} \ll (\log x)^{\delta} x^{1-s},
\end{equation*}
where $1-s>0$. The case $s=1$ is well-known. If $s>1$, then the corresponding series is convergent.
\end{proof}

\begin{proof}[Proof of Theorem {\rm \ref{Th_1_general}}]
We use identity \eqref{Gf3} and the known estimate
\begin{equation} \label{sum_2_omega}
\sum_{n \le x} 2^{\omega(n)} = \frac{x}{\zeta(2)}(\log x+ C) + S(x),
\end{equation}
where $C$ is defined by \eqref{def_C} and $S(x)\ll x^{1/2}$ (see \cite{GV1966}), that can be improved to $S(x)\ll x^{1/2}\lambda(x)$
with $\lambda(x)$ given by \eqref{lambda} (see \cite[Th.\ 3.1]{SurSiv1971}).

We deduce by standard arguments that
\begin{equation*}
\sum_{mn\le x} f((m,n)) = \sum_{d^2c\le x} f(d) 2^{\omega(c)}= \sum_{d\le x^{1/2}} f(d) \sum_{c\le x/d^2} 2^{\omega(c)}
\end{equation*}
\begin{equation*}
=  \sum_{d\le x^{1/2}} f(d) \left(\frac1{\zeta(2)}\frac{x}{d^2}(\log \frac{x}{d^2}+ C) + S\left(\frac{x}{d^2}\right) \right)
\end{equation*}
\begin{equation} \label{3_terms}
= \frac{x}{\zeta(2)}\left((\log x + C) \sum_{d\le x^{1/2}} \frac{f(d)}{d^2} - 2 \sum_{d\le x^{1/2}} \frac{f(d)\log d}{d^2}\right) +
\sum_{d\le x^{1/2}} f(d) S\left(\frac{x}{d^2}\right).
\end{equation}

Here
\begin{equation*}
\sum_{d\le x^{1/2}} \frac{f(d)}{d^2} = \sum_{d=1}^{\infty} \frac{f(d)}{d^2} + A_f(x),
\end{equation*}
where the series converges absolutely by the given assumption on $f$, and
\begin{equation*}
A_f(x)=  \sum_{d>x^{1/2}} \frac{|f(d)|}{d^2}\ll \sum_{d>x^{1/2}} \frac{(\log d)^{\delta}}{d^{2-\beta}} \ll x^{(\beta-1)/2} (\log x)^{\delta}
\end{equation*}
by using Lemma \ref{Lemma_n>x}, where $2-\beta>1$, leading to the error $x(\log x) x^{(\beta-1)/2} (\log x)^{\delta}=x^{(\beta+1)/2}(\log x)^{\delta+1}$.

Furthermore,
\begin{equation*}
\sum_{d\le x^{1/2}} \frac{f(d)\log d}{d^2} = \sum_{d=1}^{\infty} \frac{f(d)\log d}{d^2} + B_f(x),
\end{equation*}
where the series converges absolutely and
\begin{equation*}
B_f(x)=  \sum_{d>x^{1/2}} \frac{|f(d)|\log d}{d^2}\ll \sum_{d>x^{1/2}} \frac{(\log d)^{\delta+1}}{d^{2-\beta}} \ll x^{(\beta-1)/2} (\log x)^{\delta+1}
\end{equation*}
by using Lemma \ref{Lemma_n>x} again, giving the same error $x^{(\beta+1)/2}(\log x)^{\delta+1}$.

Finally, we estimate the last term in \eqref{3_terms}, namely the sum
\begin{equation} \label{def_T_x}
T(x):= \sum_{d\le x^{1/2}} f(d) S\left(\frac{x}{d^2}\right).
\end{equation}

We have by using that $S(x)\ll x^{1/2}$,
\begin{equation*}
T(x) \ll x^{1/2} \sum_{d\le x^{1/2}} \frac{|f(d)|}{d} \ll x^{1/2} \sum_{d\le x^{1/2}}
\frac{(\log d)^{\delta}}{d^{1-\beta}}
\end{equation*}
\begin{equation*}
\ll
\begin{cases}
x^{(\beta+1)/2} (\log x)^{\delta}, & \text{ if $0<\beta <1$}, \\
x^{1/2} (\log x)^{\delta+1}, & \text{ if $\beta=0$, $\delta \ne -1$}, \\
x^{1/2} \log \log x, & \text{ if $\beta=0$, $\delta=-1$},
\end{cases}
\end{equation*}
by Lemma \ref{Lemma_n<x}.

If $\beta <0$, then we use that $S(x)\ll x^{1/2}\lambda(x)$ for $x\ge 3$ and $S(x)\ll 1$ for $1\le x <3$. Note that for $x\ge 3$, $\lambda(x): =
e^{-c(\log x)^{3/5}(\log \log x)^{-1/5}}$ is decreasing, but $x^{\varepsilon} \lambda(x)$ is increasing for every $\varepsilon >0$,
if $c>0$ is sufficiently small. We split the sum $T(x)$ defined by \eqref{def_T_x} in two sums: $T(x)= T_1(x)+T_2(x)$, according to
$x/d^2\ge 3$ and $x/d^2<3$, respectively. We deduce
\begin{equation*}
T_1(x):= \sum_{\substack{d\le x^{1/2}\\ x/d^2\ge 3}} f(d) S\left(\frac{x}{d^2}\right) \ll x^{1/2} \sum_{d\le (x/3)^{1/2}} \frac{(\log d)^{\delta}}{d^{1-\beta}} \lambda \left(\frac{x}{d^2} \right)
\end{equation*}
\begin{equation*}
= x^{1/2-\varepsilon} \sum_{d\le (x/3)^{1/2}} \frac{(\log d)^{\delta}}{d^{1-\beta-2\varepsilon}} \left(\frac{x}{d^2} \right)^{\varepsilon}
\lambda \left(\frac{x}{d^2} \right)
\end{equation*}
\begin{equation*}
\ll x^{1/2-\varepsilon} x^{\varepsilon} \lambda(x) \sum_{d=1}^{\infty} \frac{(\log d)^{\delta}}{d^{1-\beta-2\varepsilon}} \ll x^{1/2}\lambda(x),
\end{equation*}
by choosing $0<\varepsilon < - \beta/2$. Furthermore, by Lemma \ref{Lemma_n>x} (with $s=-\beta>0$),
\begin{equation*}
T_2(x):= \sum_{\substack{d\le x^{1/2}\\ x/d^2 < 3}} f(d) S\left(\frac{x}{d^2}\right) \ll
\sum_{d > (x/3)^{1/2}} d^{\beta} (\log d)^{\delta}  \ll x^{(\beta+1)/2} (\log x)^{\delta} \ll x^{1/2}\lambda(x).
\end{equation*}

Baker \cite{Bak1996} proved that under the Riemann Hypothesis for the error term $S(x)$ of estimate \eqref{sum_2_omega} one
has $S(x) \ll x^{\frac{4}{11}+\varepsilon}$, whilst Kaczorowski and Wiertelak \cite{KacWie2007} remarked that a slight modification of the treatment
in \cite{Wu2002} yields $S(x)\ll  x^{\frac{221}{608}+\varepsilon}$. This leads to the desired improvement of the error.
\end{proof}

Now we will prove Theorem \ref{Th_2_general}. We need the following Lemmas.

\begin{lemma} \label{Lemma_p>x} If $\eta \ge 0$ and $s>1$ are real numbers, then
\begin{equation*}
\sum_{p>x} \frac{(\log p)^{\eta}}{p^s}\ll \frac{(\log x)^{\eta-1}}{x^{s-1}}.
\end{equation*}
\end{lemma}

\begin{proof} We have, by using Riemann-Stieltjes integration, integration by parts and the Chebyshev estimate $\pi(x)\ll x/\log x$,
\begin{equation*}
\sum_{p> x} \frac{(\log p)^{\eta}}{p^s} = \int_x^{\infty} \frac{(\log t)^{\eta}}{t^s}\, d(\pi(t)) =
\left[\frac{(\log t)^{\eta}}{t^s} \pi(t) \right]_{t=x}^{t=\infty} - \int_x^{\infty} \left(\frac{(\log t)^\eta}{t^s}\right)' \pi(t)\, dt
\end{equation*}
\begin{equation*}
\ll \frac{(\log x)^{\eta-1}}{x^{s-1}} + \int_x^{\infty} \frac{(\log t)^{\eta-1}}{t^s}\, dt.
\end{equation*}

Integration by parts, again, gives
\begin{equation*}
\int_x^{\infty} \frac{(\log t)^{\eta-1}}{t^s}\, dt \ll \frac{(\log x)^{\eta-1}}{x^{s-1}} + \int_x^{\infty} \frac{(\log t)^{\eta-2}}{t^s}\, dt,
\end{equation*}
and repeated applications of the latter estimate conclude the result.
\end{proof}

\begin{lemma} \label{Lemma_p<x} If $\eta \ge 0$ and $0\le s< 1$ are real numbers, then
\begin{equation*}
  \sum_{p\le x} \frac{(\log p)^\eta}{p^s}\ll \frac{(\log x)^{\eta-1}}{x^{s-1}}.
\end{equation*}
\end{lemma}

\begin{proof} Similar to the previous proof, by using Riemann-Stieltjes integration,
\begin{equation*}
\sum_{p\le x} \frac{(\log p)^\eta}{p^s} = \int_2^x \frac{(\log t)^\eta}{t^s}\, d(\pi(t)) =
\left[\frac{(\log t)^\eta}{t^s} \pi(t) \right]_{t=2}^{t=x} - \int_2^x \left(\frac{(\log t)^\eta}{t^s}\right)' \pi(t)\, dt
\end{equation*}
\begin{equation*}
\ll \frac{(\log x)^{\eta-1}}{x^{s-1}} + \int_2^x \frac{(\log t)^{\eta-1}}{t^s} \, dt \ll \frac{(\log x)^{\eta-1}}{x^{s-1}}.
\end{equation*}
\end{proof}

\begin{proof}[Proof of Theorem {\rm \ref{Th_2_general}}]
Now we use identity \eqref{Gf2} and the well-known estimate on $\tau(n)$,
\begin{equation} \label{tau_estimate}
\sum_{n \le x} \tau(n) = x(\log x+ C_1) + O(x^{\theta+\varepsilon}),
\end{equation}
where $C_1=2\gamma-1$ and $1/4<\theta < 1/2$. Put $\theta_1=\theta+\varepsilon$. We note that the final error term of our asymptotic formula
will not depend on $\theta_1$ and it will be enough to use $\theta_1< 1/2$. We have
\begin{equation*}
\sum_{mn\le x} f((m,n)) = \sum_{d^2c\le x} (\mu*f)(d) \tau(c) = \sum_{d\le x^{1/2}} (\mu*f)(d) \sum_{c\le x/d^2} \tau(c)
\end{equation*}
\begin{equation*}
=  \sum_{d\le x^{1/2}} (\mu*f)(d) \left(\frac{x}{d^2}(\log \frac{x}{d^2}+ C_1) + O((\frac{x}{d^2})^{\theta_1}) \right)
\end{equation*}
\begin{equation*}
= x \left((\log x + C_1) \sum_{d\le x^{1/2}} \frac{(\mu*f)(d)}{d^2} - 2 \sum_{d\le x^{1/2}} \frac{(\mu*f)(d)\log d}{d^2}\right) +
O(x^{\theta_1} \sum_{d\le x^{1/2}} \frac{|(\mu*f)(d)|}{d^{2\theta_1}}).
\end{equation*}

Here
\begin{equation*}
A:= \sum_{d\le x^{1/2}} \frac{(\mu*f)(d)}{d^2} = \sum_{\substack{p^{\nu}\le x\\ p\in Q\\ \nu \in S}} \frac{g(p)}{p^{2\nu}}
= \sum_{\substack{p\le x^{1/2}\\ p\in Q}} g(p) \sum_{\substack{1\le \nu \le m\\ \nu \in S}} \frac1{p^{2\nu}}
\end{equation*}
\begin{equation} \label{A}
= \sum_{\substack{p\le x^{1/2}\\ p\in Q}} g(p) \left(H_S(p) - \sum_{\substack{ \nu \ge m+1\\ \nu \in S}} \frac1{p^{2\nu}}\right),
\end{equation}
where $m=:\lfloor \frac{\log x}{2\log p}\rfloor$, and for every prime $p$,
\begin{equation} \label{H_S}
\frac1{p^2}\le H_S(p):=  \sum_{\nu \in S}^{\infty} \frac1{p^{2\nu}}\le \sum_{\nu=1}^{\infty} \frac1{p^{2\nu}} = \frac1{p^2-1},
\end{equation}
using that $1\in S$. Here
\begin{equation} \label{A_A}
\sum_{\substack{p\le x^{1/2}\\ p\in Q}} g(p) H_S(p) = \sum_{p\in Q} g(p)H_S(p) - \sum_{\substack{p>x^{1/2}\\ p\in Q}} g(p) H_S(p),
\end{equation}
where the series is absolutely convergent by the condition $g(p)\ll (\log p)^\eta$ and by \eqref{H_S}, and the last sum is
\begin{equation*}
\ll  \sum_{p>x^{1/2}} \frac{(\log p)^\eta}{p^2-1}\ll  \sum_{p>x^{1/2}} \frac{(\log p)^\eta}{p^2} \ll \frac{(\log x)^{\eta-1}}{x^{1/2}}
\end{equation*}
by Lemma \ref{Lemma_p>x}. Also,
\begin{equation} \label{A_1}
A_1:= \sum_{\substack{p\le x^{1/2}\\ p\in Q}}  g(p) \sum_{\substack{\nu \ge m+1 \\ \nu \in S}} \frac1{p^{2\nu}}
\ll \sum_{p\le x^{1/2}} (\log p)^\eta \sum_{\nu \ge m+1} \frac1{p^{2\nu}}
\end{equation}
\begin{equation} \label{sum_p}
=  \sum_{p\le x^{1/2}} \frac{(\log p)^\eta}{p^{2m}(p^2-1)}.
\end{equation}

By the definition of $m$ we have $m> \frac{\log x}{2\log p}-1$, hence $p^{2m}>\frac{x}{p^2}$. Thus the sum in \eqref{sum_p} is
\begin{equation} \label{est_need}
\le  \frac1{x} \sum_{p\le x^{1/2}} \frac{p^2(\log p)^\eta}{p^2-1}\ll  \frac1{x} \sum_{p\le x^{1/2}} (\log p)^\eta \le  \frac1{x}
(\log x)^\eta \pi(x^{1/2}),
\end{equation}
hence
\begin{equation} \label{A_2}
A_1 \ll \frac{(\log x)^{\eta-1}}{x^{1/2}},
\end{equation}
using $\eta \ge 0$ and the estimate $\pi(x^{1/2})\ll \frac{x^{1/2}}{\log x}$.

We deduce by \eqref{A}, \eqref{A_A}, \eqref{A_1} and \eqref{A_2} that
\begin{equation*}
A= \sum_{p\in Q} g(p)H_S(p) + O\left(\frac{(\log x)^{\eta-1}}{x^{1/2}}\right),
\end{equation*}
which leads to the error term $\ll x^{1/2} (\log x)^\eta$.

In a similar way,
\begin{equation*}
B:= \sum_{d\le x^{1/2}} \frac{(\mu*f)(d)\log d}{d^2} = \sum_{\substack{p^{\nu}\le x\\ p\in Q\\ \nu \in S}} \frac{g(p)\log p^{\nu}}{p^{2\nu}}
= \sum_{\substack{p\le x^{1/2}\\ p\in Q}} g(p)\log p \sum_{\substack{1\le \nu \le m\\ \nu \in S}} \frac{\nu}{p^{2\nu}}
\end{equation*}
\begin{equation*}
= \sum_{\substack{p\le x^{1/2}\\ p\in Q}} g(p) \log p \left(K_S(p) - \sum_{\substack{ \nu \ge m+1\\ \nu \in S}} \frac{\nu}{p^{2\nu}}\right),
\end{equation*}
where for every prime $p$,
\begin{equation} \label{K_S}
\frac1{p^2}\le K_S(p):=  \sum_{\substack{\nu=1\\ \nu \in S}}^{\infty} \frac{\nu}{p^{2\nu}} \le \sum_{\nu=1}^{\infty} \frac{\nu}{p^{2\nu}} =
\frac{p^2}{(p^2-1)^2}\ll \frac1{p^2},
\end{equation}
since $1\in S$. We write
\begin{equation*}
\sum_{\substack{p\le x^{1/2}\\ p\in Q}} g(p) (\log p) K_S(p) = \sum_{p\in Q} g(p) (\log p) K_S(p) - \sum_{\substack{p>x^{1/2}\\ p\in Q}} g(p)
(\log p) K_S(p),
\end{equation*}
where the series is absolutely convergent by \eqref{K_S}, and the last sum is
\begin{equation*}
\ll  \sum_{p>x^{1/2}} \frac{(\log p)^{\eta+1}}{p^2}\ll \frac{(\log x)^\eta}{x^{1/2}}
\end{equation*}
by Lemma \ref{Lemma_p>x}. Also,
\begin{equation} \label{B_1}
B_1:= \sum_{\substack{p\le x^{1/2}\\ p\in Q}}  g(p)\log p \sum_{\substack{\nu \ge m+1 \\ \nu \in S}} \frac{\nu}{p^{2\nu}}
\ll \sum_{p\le x^{1/2}} (\log p)^{\eta+1} \sum_{\nu \ge m+1} \frac{\nu}{p^{2\nu}},
\end{equation}
where
\begin{equation*}
\sum_{\nu \ge m+1} \frac{\nu}{p^{2\nu}}= \frac{p^2}{(p^2-1)^2}\left(\frac{m}{p^{2m+2}}- \frac{m+1}{p^{2m}} \right)\ll
\frac1{p^2}\frac{m}{p^{2m}}\ll \frac{\log x}{x\log p},
\end{equation*}
using that $p^{2m}>\frac{x}{p^2}$ and $m\ll \frac{\log x}{\log p}$. This gives that for the sum $B_1$ in \eqref{B_1},
\begin{equation*}
B_1 \le  \frac{\log x}{x} \sum_{p\le x^{1/2}} (\log p)^\eta \ll  \frac{(\log x)^\eta}{x^{1/2}},
\end{equation*}
see \eqref{est_need}. Putting all of this together we obtain that
\begin{equation*}
B= \sum_{p\in Q} g(p)(\log p) K_S(p) + O\left(\frac{(\log x)^\eta}{x^{1/2}}\right),
\end{equation*}
which leads to the error term $\ll x^{1/2} (\log x)^\eta$, the same as above.

Finally, \begin{equation*}
C:= x^{\theta_1} \sum_{d\le x^{1/2}} \frac{|(\mu*f)(d)|}{d^{2\theta_1}} =
x^{\theta_1} \sum_{\substack{p^{\nu}\le x^{1/2} \\ p\in Q\\ \nu \in S}}  \frac{g(p)}{p^{2\nu \theta_1}}
\le x^{\theta_1} \sum_{p\le x^{1/2}} (\log p)^\eta \sum_{\nu =1}^{\infty} \frac1{p^{2\nu \theta_1}}
\end{equation*}
\begin{equation*}
\ll x^{\theta_1} \sum_{p\le x^{1/2}} \frac{(\log p)^\eta}{p^{2\theta_1}}\ll x^{1/2} (\log x)^{\eta-1}
\end{equation*}
by Lemma \ref{Lemma_p<x}. This finishes the proof. \end{proof}

\begin{proof}[Proof of Theorem {\rm \ref{Theorem_sum_lcm_m_n}}] We have
\begin{equation*}
\sum_{ab\le x} [a,b] = \sum_{n\le x} n \sum_{ab=n} \frac1{(a,b)},
\end{equation*}
and partial summation applied to estimate \eqref{estimate_1_per_(a,b)} gives the result.
\end{proof}

\begin{proof}[Proof of Theorem {\rm \ref{Th_log_lcm}}] We have, by using identity \eqref{GLf},
\begin{equation*}
\sum_{mn\le x} \log [m,n] = \sum_{n\le x} \tau(n)\log n - \sum_{mn\le x} \log (m,n).
\end{equation*}

We obtain by partial summation on \eqref{tau_estimate} that
\begin{equation*}
\sum_{n \le x} \tau(n) \log n = x ((\log x)^2 +(2\gamma-2)\log x + 2-2\gamma) + O(x^{\theta+\varepsilon}),
\end{equation*}
which can be combined with formula \eqref{log_asymp} on $\sum_{mn\le x} \log (m,n)$.
\end{proof}

\begin{proof}[Proof of Theorem {\rm \ref{Th_omega_lcm}}] By identity \eqref{GLf_addit} we have
\begin{equation*}
\sum_{mn\le x} \omega([m,n]) = 2 \sum_{mn\le x} \omega(n) - \sum_{mn\le x} \omega((m,n)),
\end{equation*}
where
\begin{equation*}
 \sum_{mn\le x} \omega(n)  =  \sum_{n\le x} \omega(n) \sum_{m\le x/n} 1 = x\sum_{n\le x} \frac{\omega(n)}{n} + O\left (\sum_{n\le x} \omega(n)\right).
 \end{equation*}

As well known,
\begin{equation*}
\sum_{n\le x} \omega(n)  =  x \log \log x  + M x +  O\left (\frac{x}{\log x}\right),
\end{equation*}
where
\begin{equation*}
 M= \gamma + \sum_p \left(\log \left(1-\frac1{p} \right) +\frac1{p}\right) \doteq 0.261497
 \end{equation*}
 is the Mertens constant. By partial summation we have
 \begin{equation*}
\sum_{n\le x} \frac{\omega(n)}{n}  =  (\log x)(\log \log x)  + (M-1)\log x +  O(\log \log x).
\end{equation*}
 
 Using also formula \eqref{omega_asymp} on $\sum_{mn\le x} \omega((m,n))$ we deduce \eqref{form_omega_lcm} with error term $O(x\log \log x)$. For a different approach observe that 
 \begin{equation*}
\sum_{mn=k} \omega([m,n]) = \sum_{mn=k} \omega(k) = \omega(k) \tau(k),
\end{equation*}
since if $mn=k$, then the prime factors of  $[m,n]$ coincide with the prime factors of $k$, so $\omega([m,n])=\omega(k)$.

Hence
\begin{equation*}
\sum_{mn\le x} \omega([m,n]) = \sum_{n\le x} \omega(n) \tau(n),
\end{equation*}
and the asymptotic formula for the latter sum, established by De Koninck and Mercier \cite[Th.\ 9]{deKMer1977} using analytic arguments, gives 
the error $O(x)$. Note that in \cite{deKMer1977} the constant \eqref{const_K_omega} is given in a different form.
\end{proof}

\begin{proof}[Proof of Theorem {\rm \ref{Th_Omega_lcm}}] By using that
\begin{equation*}
\sum_{n\le x} \Omega(n)  =  x \log \log x  + \left(M + \sum_p \frac1{p(p-1)}\right) x +  O\left (\frac{x}{\log x}\right),
\end{equation*}
the first approach in the proof of Theorem \ref{Th_omega_lcm} applies, and gives the error $O(x\log \log x)$. 
However, to obtain the better error term $O(x)$ we proceed as follows. 
The function $\Omega(n)$ is completely additive, hence by \eqref{GLf},
\begin{equation*}
\sum_{mn\le x} \Omega([m,n]) = \sum_{n\le x} \Omega(n) \tau(n) - \sum_{mn\le x} \Omega((m,n)).
\end{equation*}

It follows from the general result by De Koninck and Mercier \cite[Th.\ 8]{deKMer1977}, applied to the function $\Omega(n)$ that
\begin{equation*}
\sum_{n\le x} \Omega(n) \tau(n) = 2x(\log x)(\log \log x) + K_{\Omega} x\log x+ O(x),
\end{equation*}
where the constant $K_{\Omega}$ is defined by \eqref{const_K_Omega}. Now using also estimate \eqref{Omega_asymp} on $\sum_{mn\le x} \Omega((m,n))$ finishes the proof of \eqref{form_Omega_lcm}.
\end{proof}

\begin{proof}[Proof of Theorem {\rm \ref{Th_tau_lcm}}]
We show that
\begin{equation} \label{id_h}
h(n):=\sum_{ab=n} \tau([a,b])=\sum_{dk=n} \psi(d)\tau^2(k),
\end{equation}
where the function $\psi$ is multiplicative and $\psi(p^{\nu})= (-1)^{\nu-1}(\nu-1)$ for every prime power $p^{\nu}$ ($\nu \ge 0$).

This can be done by multiplicativity and computing the values of both sides for prime powers. However, we present here a different approach,
based on identity \eqref{Lf2}. The Dirichlet series of $h(n)$ is
\begin{equation} \label{H}
H(z):= \sum_{n=1} \frac{h(n)}{n^z}= \sum_{d^2k=1}^{\infty} \frac{\tau(dk)2^{\omega(k)}}{d^{2z}k^z} = \sum_{k=1}^{\infty} \frac{2^{\omega(k)}}{k^z}
\sum_{d=1}^{\infty} \frac{\tau(dk)}{d^{2z}}.
\end{equation}

If $f$ is any multiplicative function and $k=\prod_p p^{\nu_p(k)}$ a positive integer, then
\begin{equation*}
\sum_{n=1}^{\infty} \frac{f(kn)}{n^z} = \prod_p \sum_{\nu=0}^{\infty} \frac{f(p^{\nu+\nu_p(k)})}{p^{\nu z}}.
\end{equation*}

If $f(n)=\tau(n)$, then this gives (see Titchmarsh \cite[Sect.\ 1.4.2]{Tit1986})
\begin{equation} \label{tau_kn}
\sum_{n=1}^{\infty} \frac{\tau(kn)}{n^z} = \zeta^2(z) \prod_p \left(\nu_p(k)+1 - \frac{\nu_p(k)}{p^z} \right).
\end{equation}

By inserting \eqref{tau_kn} into \eqref{H} we deduce
\begin{equation*}
H(z)= \zeta^2(2z) \sum_{k=0}^{\infty} \frac{2^{\omega(k)}h_z(k)}{k^z},
\end{equation*}
where $h_z$ is the multiplicative function given by $h_z(k)= \prod_p \left(\nu_p(k)+1 -\frac{\nu_p(k)}{p^{2z}} \right)$, depending on $z$.
Therefore, by the Euler product formula,
\begin{equation*}
H(z)= \zeta^2(2z) \prod_p \left(1+ \sum_{\nu=1}^{\infty} \frac{2}{p^{\nu z}} \left(\nu +1-\frac{\nu}{p^{2z}} \right)\right)
\end{equation*}
\begin{equation*}
= \zeta^2(2z) \prod_p \left(1+2(1-\frac1{p^{2z}}) \sum_{\nu=1}^{\infty} \frac{\nu+1}{p^{\nu z}} + \frac{2}{p^{2z}} \sum_{\nu=1}^{\infty}
\frac1{p^{\nu z}} \right)
\end{equation*}
\begin{equation*}
= \zeta^2(2z) \prod_p \left(1+2\left(1-\frac1{p^{2z}}\right)\left(\left(1-\frac1{p^z}\right)^{-2}-1\right)
+ \frac{2}{p^{3z}}\left(1-\frac1{p^z}\right)^{-1}\right)
\end{equation*}
\begin{equation*}
= \zeta^2(2z) \prod_p \left(1+2\left(1+\frac1{p^z}\right) \left(1-\frac1{p^z}\right)^{-1} - 2 \left(1-\frac1{p^{2z}}\right)
+ \frac{2}{p^{3z}}\left(1-\frac1{p^z}\right)^{-1}\right)
\end{equation*}
\begin{equation*}
= \zeta(z) \zeta^2(2z) \prod_p \left(1+ \frac1{p^z}\right)\left(1+\frac{2}{p^z} \right)
\end{equation*}
\begin{equation*}
= \zeta^2(z) \zeta(2z) \prod_p \left(1+\frac{2}{p^z} \right),
\end{equation*}
which can be written as
\begin{equation*}
H(z) = \frac{\zeta^4(z)}{\zeta(2z)} G(z),
\end{equation*}
where
\begin{equation} \label{H_inf}
G(z) = \prod_p \left(1-\frac1{(p^z+1)^2} \right) = \prod_p \sum_{\nu=0}^{\infty} \frac{(-1)^{\nu-1}(\nu-1)}{p^{\nu z}}.
\end{equation}

Here $\frac{\zeta^4(z)}{\zeta(2z)}=\sum_{n=1}^{\infty} \frac{\tau^2(n)}{n^z}$, as well known. This proves identity \eqref{id_h}.

The infinite product in \eqref{H_inf} is absolutely convergent for $\Re z >1/2$, Using Ramanujan's formula
\begin{equation*}
\sum_{n\le x} \tau^2(n)= x(a(\log x)^3 +b(\log x)+ c \log x+d) +O(x^{1/2+\varepsilon}),
\end{equation*}
where $a=1/\pi^2$, the convolution method leads to asymptotic formula \eqref{tau_lcm}. The main coefficient is
$(1/\pi^2)G(1)= (1/\pi^2) \prod_p \left(1-\frac1{(p+1)^2} \right)$. See the similar proof of \cite[Th. \1]{Tot2010}.
\end{proof}

\begin{proof}[Proof of Theorem {\rm \ref{Th_gcd_per_lcm}}] We have
\begin{equation} \label{sssum}
\sum_{ab\le x} \frac{(a,b)}{[a,b]}= \sum_{n\le x} \frac1{n} \sum_{ab=n} (a,b)^2.
\end{equation}

Let $f(n)=n^2$. Then $(\mu*f)(n)=\phi_2(n)=n^2 \prod_{p\mid n} (1-1/p^2)$ is the Jordan function of order $2$. Here Theorems \ref{Th_1_general}
and \ref{Th_2_general} cannot be applied. However, the estimate
\begin{equation*}
\sum_{n\le x} \phi_2(n)= \frac{x^3}{3\zeta(3)}+O(x^2),
\end{equation*}
is well-known, and using identity \eqref{Gf2} we deduce that
\begin{equation*}
\sum_{ab\le x} (a,b)^2 = \sum_{d^2k\le x} \phi_2(d)\tau(k) = \sum_{k\le x} \tau(k) \sum_{d\le (x/k)^{1/2}} \phi_2(d)
\end{equation*}
\begin{equation*}
=\frac{x^{3/2}}{3\zeta(3)} \sum_{k\le x} \frac{\tau(k)}{k^{3/2}}+ O\left(x\sum_{k\le x} \frac{\tau(k)}{k} \right)
\end{equation*}
\begin{equation*}
=\frac{\zeta^2(3/2)}{3\zeta(3)}x^{3/2} + O\left(x (\log x)^2 \right).
\end{equation*}

Now, taking into account \eqref{sssum}, partial summation concludes formula \eqref{final_est}.
\end{proof}

\medskip

{\sl Acknowledgements}: The authors thank the referee for valuable suggestions and remarks.
The second author was supported by the European Union, co-financed by the European Social Fund EFOP-3.6.1.-16-2016-00004.

\end{document}